\documentclass[12pt]{article}
\usepackage{amssymb}
\usepackage{amsthm}
\usepackage{amsmath}
\usepackage{graphicx}
\usepackage{enumerate}

\newtheorem{theorem}{Theorem}%[section]

\newtheorem{lemma}[theorem]{Lemma}

\newtheorem{example}[theorem]{Example}
\newtheorem{corollary}[theorem]{Corollary}
\title{Extensions of posets with an antitone involution to residuated structures}
\author{Ivan~Chajda, Miroslav~Kola\v r\'ik and Helmut~L\"anger}
\date{}
\begin{document}
\footnotetext[1]{Support of the research of the first and the third author by the Austrian Science Fund (FWF), project I~4579-N, and the Czech Science Foundation (GA\v CR), project 20-09869L, entitled ``The many facets of orthomodularity'', as well as by \"OAD, project CZ~02/2019, entitled ``Function algebras and ordered structures related to logic and data fusion'', and, concerning the first author, by IGA, project P\v rF~2020~014, is gratefully acknowledged.}
\maketitle
\begin{abstract}
We prove that every not necessarily bounded poset $\mathbf P=(P,\leq,{}')$ with an antitone involution can be extended to a residuated poset $\mathbb E(\mathbf P)=(E(P),\leq,\odot,\rightarrow,1)$ where $x'=x\rightarrow0$ for all $x\in P$. If $\mathbf P$ is a lattice with an antitone involution then $\mathbb E(\mathbf P)$ is a lattice, too. We show that a poset can be extended to a residuated poset by means of a finite chain and that a Boolean algebra $(B,\vee,\wedge,{}',p,q)$ can be extended to a residuated lattice $(Q,\vee,\wedge,\odot,\rightarrow,1)$ by means of a finite chain in such a way that $x\odot y=x\wedge y$ and $x\rightarrow y=x'\vee y$ for all $x,y\in B$.
\end{abstract}

{\bf AMS Subject Classification:} 03B52, 06A11, 06B05, 03B47

{\bf Keywords:} Poset, bounded poset, residuated poset, residuated lattice, antitone involution, extension

Residuated posets in general and residuated lattices in particular form an algebraic axiomatization of certain substructural logics (see e.g.\ \cite{CF}, \cite C, \cite J and \cite{JT} and references therein), especially of fuzzy logic, see \cite B for details. Residuated lattices were studied for a long time starting with the pioneering paper by Ward and Dilworth \cite{WD}, see also \cite{C18} and \cite{CL19b}. Posets and lattices with an antitone involution can serve as a suitable model of such a logic because this involution can be considered as a negation and hence these logics satisfy the double negation law, see \cite{C18}. Let us mention that a kind of residuated posets were studied also in \cite{CL18}. Moreover, residuated structures derived from semirings were treated in \cite{CL19a} and \cite J.

Recall that a {\em poset with an antitone involution} is an ordered triple $(P,\leq,{}')$ such that $(P,\leq)$ is a poset and $'$ is a unary operation on $P$ satisfying
\begin{itemize}
\item if $a\leq b$ then $b'\leq a'$,
\item $a''\approx a$
\end{itemize}
for all $a,b\in P$. Recall further that a {\em residuated poset} is an ordered quintuple $(P,\leq,\odot,\rightarrow,1)$ such that
\begin{itemize}
\item $(P,\leq,1)$ is a poset with a greatest element,
\item $(P,\odot,1)$ is a commutative monoid,
\item $\rightarrow$ is a binary operation on $P$,
\item $a\odot b\leq c$ if and only if $a\leq b\rightarrow c$
\end{itemize}
for all $a,b,c\in P$. The last property is called {\em adjointness}.

Unfortunately, not every lattice $(L,\vee,\wedge,{}')$ with an antitone involution $'$ can be converted into a residuated lattice $(L,\vee,\wedge,\odot,\rightarrow,1)$ satisfying $x'=x\rightarrow0$ for all $x\in L$. For example, consider the non-modular lattice $\mathbf N_5=(N_5,\vee,\wedge)$ whose elements are $0,a,b,c,1$ where $0$ is the least and $1$ the greatest element, $a<b$ and $c$ is incomparable with $a$ and $b$, see Figure~1:
\vspace*{-4mm}
\begin{center}
\setlength{\unitlength}{7mm}
\begin{picture}(4,8)
\put(2,1){\circle*{.3}}
\put(1,3){\circle*{.3}}
\put(3,4){\circle*{.3}}
\put(1,5){\circle*{.3}}
\put(2,7){\circle*{.3}}
\put(2,1){\line(-1,2)1}
\put(2,1){\line(1,3)1}
\put(1,3){\line(0,1)2}
\put(2,7){\line(-1,-2)1}
\put(2,7){\line(1,-3)1}
\put(1.85,.25){$0$}
\put(.3,2.85){$a$}
\put(.3,4.85){$b$}
\put(3.4,3.85){$c$}
\put(1.85,7.4){$1$}
\put(1.2,-.75){{\rm Fig.\ 1}}
\end{picture}
\end{center}
\vspace*{4mm}
It is easy to see that there exists exactly one antitone involution $'$ on $(N_5,\leq)$, namely $0'=1$, $a'=b$, $b'=a$, $c'=c$ and $1'=0$. Suppose, $\mathbf N_5$ together with $'$ could be converted into a residuated poset $(N_5,\leq,\odot,\rightarrow,1)$ satisfying $x'=x\rightarrow0$ for all $x\in N_5$. Since $x\odot y\leq x$ and $x\odot y\leq y$ for all $x,y\in N_5$ (see Theorem~2.17 in \cite B) we have $c\odot a\leq0$ which implies $c\leq a\rightarrow0=a'=b$, a contradiction. Thus $\mathbf N_5$ together with $'$ cannot be converted into a residuated poset $(N_5,\leq,\odot,\rightarrow,1)$ satisfying $x'=x\rightarrow0$ for all $x\in N_5$.

Hence, it is a question whether such a lattice (or poset in general) can be extended to a residuated one by preserving the antitone involution. The aim of our paper is to show how such an extension can be constructed.

At first, we show that the unary operation $x':=x\rightarrow0$ in a residuated lattice is antitone.

\begin{lemma}\label{lem1}
Let $(P,\leq,\odot,\rightarrow,1)$ be a residuated poset and put $x':=x\rightarrow0$ for all $x\in P$ Then $x\leq x''$ for all $x\in P$ and $'$ is antitone.
\end{lemma}

\begin{proof}
Let $a,b\in P$. Then the every of the following assertions implies the next one:
\begin{align*}
         a\rightarrow0 & \leq a\rightarrow0, \\
(a\rightarrow0)\odot a & \leq0, \\
 a\odot(a\rightarrow0) & \leq0, \\
                     a & \leq(a\rightarrow0)\rightarrow0, \\
										 a & \leq a''.
\end{align*}
Moreover, every of the following assertions implies the next one:
\begin{align*}
                     a & \leq b, \\
                     a & \leq(b\rightarrow0)\rightarrow0, \\
 a\odot(b\rightarrow0) & \leq0, \\
(b\rightarrow0)\odot a & \leq0, \\
         b\rightarrow0 & \leq a\rightarrow0, \\
                    b' & \leq a'.
\end{align*}
\end{proof}

Our first extension result is as follows.

\begin{theorem}\label{th1}
Let $\mathbf P=(P,\leq,{}')$ be a poset with an antitone involution, assume $0=c_1,\ldots,c_4=1\notin P$, put $E(P):=P\cup\{c_1,\ldots,c_4\}$ and extend $\leq$ and $'$ from $P$ to $E(P)$ by $0<c_2<x<c_3<1$ for all $x\in P$ and $c_i':=c_{5-i}$ for $i=1,\ldots,4$. Define binary operations $\odot$ and $\rightarrow$ on $E(P)$ as follows:
\begin{align*}
& 0\odot x=x\odot0:=0, 1\odot x=x\odot1:=x, \\
& 0\rightarrow x=x\rightarrow1:=1, x\rightarrow0:=x', 1\rightarrow x:=x
\end{align*}
for $x\in E(P)$ and
\[
x\odot y:=\left\{
\begin{array}{ll}
0   & \text{if }x\leq y', \\
c_2 & \text{otherwise} 
\end{array}
\right.
\quad\quad\quad x\rightarrow y:=\left\{
\begin{array}{ll}
1   & \text{if }x\leq y, \\
c_3 & \text{otherwise}
\end{array}
\right.
\]
for $x,y\in E(P)\setminus\{0,1\}$. Then $\mathbb E(\mathbf P):=(E(P),\leq,\odot,\rightarrow,1)$ is a residuated poset with the antitone involution $'$ satisfying $x'=x\rightarrow0$ for all $x\in E(P)$. {\rm(}If $\mathbf P$ is already bounded then the least and greatest element of $\mathbf P$ may be identified with $c_2$ and $c_3$, respectively. If $\mathbf P$ has elements $a,b,c,d$ satisfying $a<b\leq x\leq c<d$ for all $x\in P\setminus\{a,d\}$ then $a,b,c,d$ may be identified with $c_1,\ldots,c_4$, respectively.{\rm)}
\end{theorem}

\begin{proof}
Let $a,b,c\in E(P)$. Since $x''=x$ for all $x\in P$ and $c_i''=c_{5-i}'=c_{5-(5-i)}=c_i$ for all $i=1,\ldots,4$, $(E(P),\leq,{}',0,1)$ is a bounded poset with an antitone  involution. \\
If $\{a,b,c\}\cap\{0,1\}\neq\emptyset$ then, obviously, $(a\odot b)\odot c=a\odot(b\odot c)$.
If $a,b,c\neq0,1$, $a\leq b'$ and $b\leq c'$ then $(a\odot b)\odot c=0\odot c=0=a\odot0=a\odot(b\odot c)$. \\
If $a,b,c\neq0,1$, $a\leq b'$ and $b\not\leq c'$ then $(a\odot b)\odot c=0\odot c=0=a\odot c_2=a\odot(b\odot c)$. \\
If $a,b,c\neq0,1$, $a\not\leq b'$ and $b\leq c'$ then $(a\odot b)\odot c=c_2\odot c=0=a\odot0=a\odot(b\odot c)$. \\
If $a,b,c\neq0,1$, $a\not\leq b'$ and $b\not\leq c'$ then $(a\odot b)\odot c=c_2\odot c=0=a\odot c_2=a\odot(b\odot c)$. \\
Therefore, $\odot$ is associative. Since $a\leq b'$ is equivalent to $b\leq a'$, $\odot$ is commutative. \\
If $a=0$ then $a\odot b=0$ and $a\leq b'$. \\
If $b=0$ then $a\odot b=0$ and $a\leq b'$. \\
If $a=1$ then $a\odot b=0$ and $a\leq b'$ are equivalent to $b=0$. \\
If $b=1$ then $a\odot b=0$ and $a\leq b'$ are equivalent to $a=0$. \\
If $a,b\neq0,1$ then $a\odot b=0$ is equivalent to $a\leq b'$. \\
Hence $a\odot b=0$ if and only if $a\leq b'$. \\
If $a=0$ then $a\rightarrow b=1$ and $a\leq b$. \\
If $b=0$ then $a\rightarrow b=1$ and $a\leq b$ are equivalent to $a=0$. \\
If $a=1$ then $a\rightarrow b=1$ and $a\leq b$ are equivalent to $b=1$. \\
If $b=1$ then $a\rightarrow b=1$ and $a\leq b$. \\
If $a,b\neq0,1$ then $a\rightarrow b=1$ and $a\leq b$ are equivalent. \\
Hence $a\rightarrow b=1$ if and only if $a\leq b$. \\
If $a=0$ then $a\odot b\leq c$ and $a\leq b\rightarrow c$. \\
If $b=0$ then $a\odot b\leq c$ and $a\leq b\rightarrow c$. \\
If $c=0$ then $a\odot b\leq c$ and $a\leq b\rightarrow c$ are equivalent to $a\leq b'$. \\
If $a=1$ then $a\odot b\leq c$ and $a\leq b\rightarrow c$ are equivalent to $b\leq c$. \\
If $b=1$ then $a\odot b\leq c$ and $a\leq b\rightarrow c$ are equivalent to $a\leq c$. \\
If $c=1$ then $a\odot b\leq c$ and $a\leq b\rightarrow c$. \\
If $a,b,c\neq0,1$ then $a\odot b\leq c_2$ and $c_3\leq b\rightarrow c$ and hence $a\odot b\leq c$ and $a\leq b\rightarrow c$. \\
Thus the adjointness property holds.
\end{proof}

As mentioned above, the non-modular lattice $\mathbf N_5$ with an antitone involution cannot be converted into a residuated lattice $(N_5,\vee,\wedge,\odot,\rightarrow,1)$ with an antitone involution $'$ satisfying $x'=x\rightarrow0$ for all $x\in N_5$. Using Theorem~\ref{th1}, we can extend $\mathbf N_5$ as follows.

\begin{example}
Theorem~\ref{th1} applied to $\mathbf N_5$ yields the residuated lattice depicted in Figure~2:
\vspace*{-4mm}
\begin{center}
\setlength{\unitlength}{7mm}
\begin{picture}(4,12)
\put(2,1){\circle*{.3}}
\put(2,3){\circle*{.3}}
\put(1,5){\circle*{.3}}
\put(3,6){\circle*{.3}}
\put(1,7){\circle*{.3}}
\put(2,9){\circle*{.3}}
\put(2,11){\circle*{.3}}
\put(2,1){\line(0,1)2}
\put(2,3){\line(-1,2)1}
\put(2,3){\line(1,3)1}
\put(1,5){\line(0,1)2}
\put(2,9){\line(-1,-2)1}
\put(2,9){\line(1,-3)1}
\put(2,9){\line(0,1)2}
\put(1.875,.25){$0$}
\put(2.4,2.85){$c_2$}
\put(3.4,5.85){$c$}
\put(2.4,8.85){$c_3$}
\put(1.875,11.4){$1$}
\put(.3,4.85){$a$}
\put(.3,6.85){$b$}
\put(1.2,-.75){{\rm Fig.\ 2}}
\end{picture}
\end{center}
\vspace*{4mm}
with operation tables
\[
\begin{array}{l|lllllll}
\odot & 0 & c_2 & a   & b   & c   & c_3 & 1 \\
\hline
0     & 0 & 0   & 0   & 0   & 0   & 0   & 0 \\
c_2   & 0 & 0   & 0   & 0   & 0   & 0   & c_2 \\
a     & 0 & 0   & 0   & 0   & c_2 & c_2 & a \\
b     & 0 & 0   & 0   & c_2 & c_2 & c_2 & b \\
c     & 0 & 0   & c_2 & c_2 & 0   & c_2 & c \\
c_3   & 0 & 0   & c_2 & c_2 & c_2 & c_2 & c_3 \\
1     & 0 & c_2 & a   & b   & c   & c_3 & 1
\end{array}
\quad\quad\quad
\begin{array}{l|lllllll}
\rightarrow & 0   & c_2 & a   & b   & c   & c_3 & 1 \\
\hline
0           & 1   & 1   & 1   & 1   & 1   & 1   & 1 \\
c_2         & c_3 & 1   & 1   & 1   & 1   & 1   & 1 \\
a           & b   & c_3 & 1   & 1   & c_3 & 1   & 1 \\
b           & a   & c_3 & c_3 & 1   & c_3 & 1   & 1 \\
c           & c   & c_3 & c_3 & c_3 & 1   & 1   & 1 \\
c_3         & c_2 & c_3 & c_3 & c_3 & c_3 & 1   & 1 \\
1           & 0   & c_2 & a   & b   & c   & c_3 & 1
\end{array}
\]
Observe there is only one possibility for the antitone involution.
\end{example}

Let us note that if $\mathbf P=(P,\leq,{}')$ is a finite chain containing at least three elements (with unique antitone involution) then, using the construction from Theorem~\ref{th1}, $\mathbf P$ can be converted into a residuated chain $\mathbb E(\mathbf P)=(P,\leq,\odot,\rightarrow,1)$ satisfying $x'=x\rightarrow0$ for all $x\in P$.

\begin{corollary}\label{cor1}
If $(C,\leq)=(\{c_1,\ldots,c_n\},\leq)$ is a finite chain $0=c_1<c_2<c_3<\cdots<c_n=1$ with $n\geq3$ elements,
\begin{align*}
& 0\odot c_i=c_i\odot0:=0, 1\odot c_i=c_i\odot1:=c_i, \\
& 0\rightarrow c_i=c_i\rightarrow1:=1, c_i\rightarrow0:=c_{n+1-i}, 1\rightarrow c_i:=c_i
\end{align*}
for $i=1,\ldots,n$ and
\[
c_i\odot c_j:=\left\{
\begin{array}{ll}
0   & \text{if }i+j\leq n+1, \\
c_2 & \text{otherwise} 
\end{array}
\right.
\quad\quad\quad c_i\rightarrow c_j:=\left\{
\begin{array}{ll}
1       & \text{if }i\leq j, \\
c_{n-1} & \text{otherwise}
\end{array}
\right.
\]
for $i,j=2,\ldots,n-1$ then $(C,\leq,\odot,\rightarrow,1)$ is a residuated lattice where $x':=x\rightarrow0$ is an antitone involution.
\end{corollary}

\begin{example}\label{ex1}
Theorem~\ref{th1} or Corollary~\ref{cor1} applied to the five-element chain yields the residuated lattice depicted in Figure~3:
\vspace*{-8mm}
\begin{center}
\setlength{\unitlength}{7mm}
\begin{picture}(2,9)
\put(1,0){\circle*{.3}}
\put(1,2){\circle*{.3}}
\put(1,4){\circle*{.3}}
\put(1,6){\circle*{.3}}
\put(1,8){\circle*{.3}}
\put(1,0){\line(0,1)8}
\put(1.4,-.15){$0$}
\put(1.4,1.85){$c_2$}
\put(1.4,3.85){$c_3$}
\put(1.4,5.85){$c_4$}
\put(1.4,7.85){$1$}
\put(.2,-.95){{\rm Fig.\ 3}}
\end{picture}
\end{center}
\vspace*{4mm}
with operation tables
\[
\begin{array}{l|lllll}
\odot & 0 & c_2 & c_3 & c_4 & 1 \\
\hline
0     & 0 & 0   & 0   & 0   & 0 \\
c_2   & 0 & 0   & 0   & 0   & c_2 \\
c_3   & 0 & 0   & 0   & c_2 & c_3 \\
c_4   & 0 & 0   & c_2 & c_2 & c_4 \\
1     & 0 & c_2 & c_3 & c_4 & 1
\end{array}
\quad\quad\quad
\begin{array}{l|lllll}
\rightarrow & 0   & c_2 & c_3 & c_4 & 1 \\
\hline
0           & 1   & 1   & 1   & 1   & 1 \\
c_2         & c_4 & 1   & 1   & 1   & 1 \\
c_3         & c_3 & c_4 & 1   & 1   & 1 \\
c_4         & c_2 & c_4 & c_4 & 1   & 1 \\
1           & 0   & c_2 & c_3 & c_4 & 1
\end{array}
\]
Here $c_3$ of Theorem~\ref{th1} corresponds to $c_4$ of Example~\ref{ex1}.
\end{example}

If the poset $\mathbf P$ in question is a lattice, we can also apply the construction of $'$, $\odot$ and $\rightarrow$ from Theorem~\ref{th1} to obtain a residuated lattice $\mathbb E(\mathbf P)$. Hence, we can state the following.

\begin{corollary}
Let $\mathbf L=(L,\vee,\wedge,{}')$ be a lattice with an antitone involution $'$. Then $\mathbf L$ can be extended to a residuated lattice $\mathbb E(\mathbf L)$ with an antitone involution where the operations $'$, $\odot$ and $\rightarrow$ are constructed as in Theorem~\ref{th1} and $'$ coincides in $L$ with the original one.
\end{corollary}

Recall from \cite{C16} and \cite K that a lattice $\mathbf L=(L,\vee,\wedge,{}')$ with an antitone involution is called a {\em pseudo-Kleene algebra} if it satisfies the identities
\begin{enumerate}[(1)]
\item $x\wedge x'\leq y\vee y'$,
\item $x\wedge(x'\vee y)\approx(x\wedge x')\vee(x\wedge y)$.
\end{enumerate}
$\mathbf L$ is called a {\em Kleene algebra} if it is a distributive pseudo-Kleene algebra. In this case (2) can be omitted since it follows by distributivity. Kleene algebras and pseudo-Kleene algebras are considered as an algebraic axiomatization of a propositional logic satisfying De Morgan's laws and the double negation law, but not necessarily the law of excluded middle because the antitone involution $'$ need not be a complementation, see also \cite J. An example of a Kleene algebra is depicted in Figure~4:
\vspace*{-2mm}
\begin{center}
\setlength{\unitlength}{7mm}
\begin{picture}(4,10)
\put(2,1){\circle*{.3}}
\put(2,3){\circle*{.3}}
\put(1,5){\circle*{.3}}
\put(3,5){\circle*{.3}}
\put(2,7){\circle*{.3}}
\put(2,9){\circle*{.3}}
\put(2,3){\line(0,-1)2}
\put(2,3){\line(-1,2)1}
\put(2,3){\line(1,2)1}
\put(2,7){\line(-1,-2)1}
\put(2,7){\line(1,-2)1}
\put(2,7){\line(0,1)2}
\put(1.85,.25){$0$}
\put(2.4,2.85){$a$}
\put(.3,4.85){$b$}
\put(3.4,4.85){$b'$}
\put(2.4,6.85){$a'$}
\put(1.85,9.4){$1$}
\put(1.2,-.75){{\rm Fig.\ 4}}
\end{picture}
\end{center}
\vspace*{4mm}
A pseudo-Kleene algebra which is not a Kleene algebra is visualized in Figure~5:
\vspace*{-2mm}
\begin{center}
\setlength{\unitlength}{7mm}
\begin{picture}(4,14)
\put(2,1){\circle*{.3}}
\put(1,3){\circle*{.3}}
\put(3,4){\circle*{.3}}
\put(1,5){\circle*{.3}}
\put(2,7){\circle*{.3}}
\put(3,9){\circle*{.3}}
\put(1,10){\circle*{.3}}
\put(3,11){\circle*{.3}}
\put(2,13){\circle*{.3}}
\put(2,1){\line(-1,2)1}
\put(2,1){\line(1,3)1}
\put(1,3){\line(0,1)2}
\put(3,4){\line(-1,3)2}
\put(1,5){\line(1,2)2}
\put(3,9){\line(0,1)2}
\put(1,10){\line(1,3)1}
\put(3,11){\line(-1,2)1}
\put(1.85,.25){$0$}
\put(.3,2.85){$a$}
\put(3.4,3.85){$c$}
\put(.3,4.85){$b$}
\put(2.4,6.85){$d=d'$}
\put(3.4,8.85){$b'$}
\put(.3,9.85){$c'$}
\put(3.4,10.85){$a'$}
\put(1.85,13.4){$1$}
\put(1.2,-.75){{\rm Fig.\ 5}}
\end{picture}
\end{center}
\vspace*{4mm}
One can easily check that if $\mathbf L$ satisfies (1) and (2) then so does $\mathbb E(\mathbf L)$ and if $\mathbf L$ is distributive then also $\mathbb E(\mathbf L)$ has this property, too. Hence, we can state

\begin{corollary}
Every pseudo-Kleene algebra or every Kleene algebra can be extended to a residuated pseudo-Kleene algebra or a residuated Kleene-algebra, respectively.
\end{corollary}

Hence, also the logic axiomatized by Kleene algebras or pseudo-Kleene algebras can be extended to a kind of fuzzy logics.

Motivated by the previous, we can extend every poset with an antitone involution to a residuated poset by means of a finite chain with an even number of at least four elements. The precise formulation is as follows:

\begin{theorem}\label{th2}
Let $(P,\leq,{}')$ be a poset with an antitone involution, $n$ an integer $>1$, $0=c_1,\ldots,c_{2n}=1\notin P$ and $Q:=P\cup\{c_1,\ldots,c_{2n}\}$. Extend $\leq$ and $'$ from $P$ to $Q$ by
\[
c_1<\cdots<c_n<x<c_{n+1}<\cdots<c_{2n}
\]
for $x\in P$ and put $c_i':=c_{2n+1-i}$ for $i=1,\ldots,2n$. Define binary operations $\odot$ and $\rightarrow$ on $Q$ as follows:
\begin{align*}
& 0\odot x=x\odot0:=0, 1\odot x=x\odot1:=x, \\
& 0\rightarrow x=x\rightarrow1:=1, x\rightarrow0:=x', 1\rightarrow x:=x
\end{align*}
for $x\in Q$ and
\[
x\odot y:=\left\{
\begin{array}{ll}
0   & \text{if }x\leq y', \\
c_2 & \text{otherwise} 
\end{array}
\right.
\quad\quad\quad x\rightarrow y:=\left\{
\begin{array}{ll}
1        & \text{if }x\leq y, \\
c_{2n-1} & \text{otherwise}
\end{array}
\right.
\]
\[
c_i\odot c_j:=\left\{
\begin{array}{ll}
0   & \text{if }i+j\leq2n+1, \\
c_2 & \text{otherwise} 
\end{array}
\right.
\quad\quad\quad c_i\rightarrow c_j:=\left\{
\begin{array}{ll}
1        & \text{if }i\leq j, \\
c_{2n-1} & \text{otherwise}
\end{array}
\right.
\]
\[
c_i\odot x=x\odot c_i:=c_i\odot c_{n+1}, c_i\rightarrow x:=c_i\rightarrow c_n, x\rightarrow c_i:=c_{n+1}\rightarrow c_i
\]
for $x,y\in P$ and $i,j=2,\ldots,2n-1$. Then $(Q,\leq,\odot,\rightarrow,1)$ is a residuated poset with the antitone involution $'$ satisfying $x'=x\rightarrow0$ for all $x\in Q$.
\end{theorem}

\begin{proof}
We apply Theorem~\ref{th1} to the poset $(P\cup\{c_3,\ldots,c_{2n-2}\},\leq,{}')$ with the antitone involution $'$. The element $c_3$ of Theorem~\ref{th1} corresponds to the element $c_{2n-1}$ of Theorem~\ref{th2}. According to Theorem~\ref{th1} we have for all $x\in Q$
\begin{align*}
& 0\odot x=x\odot0=0, 1\odot x=x\odot1=x, \\
& 0\rightarrow x=x\rightarrow1=1, x\rightarrow0=x', 1\rightarrow x=x.
\end{align*}
Moreover, we have for all $x,y\in P$ and $i,j\in\{2,\ldots,2n-1\}$
\[
x\odot y=\left\{
\begin{array}{ll}
0   & \text{if }x\leq y', \\
c_2 & \text{otherwise} 
\end{array}
\right.
\quad\quad\quad x\rightarrow y=\left\{
\begin{array}{ll}
1        & \text{if }x\leq y, \\
c_{2n-1} & \text{otherwise}
\end{array}
\right.
\]
\[
c_i\odot c_j=\left\{
\begin{array}{ll}
0   & \text{if }i\leq2n+1-j, \\
c_2 & \text{otherwise} 
\end{array}
\right.
\quad\quad\quad c_i\rightarrow c_j=\left\{
\begin{array}{ll}
1        & \text{if }i\leq j, \\
c_{2n-1} & \text{otherwise}
\end{array}
\right.
\]
\[
c_i\odot x=x\odot c_i=\left\{
\begin{array}{ll}
0   & \text{if }i\leq n, \\
c_2 & \text{otherwise} 
\end{array}
\right.
\]
\[
c_i\rightarrow x=\left\{
\begin{array}{ll}
1        & \text{if }i\leq n, \\
c_{2n-1} & \text{otherwise}
\end{array}
\right.
\quad\quad\quad x\rightarrow c_i=\left\{
\begin{array}{ll}
1        & \text{if }i>n, \\
c_{2n-1} & \text{otherwise}
\end{array}
\right.
\]
\end{proof}

The poset $(Q,\leq)$ of Theorem~\ref{th2} is visualized in Figure~6:
\begin{center}
\setlength{\unitlength}{7mm}
\begin{picture}(2,6)
\put(1,0){\circle*{.3}}
\put(1,2){\circle*{.3}}
\put(1,4){\circle*{.3}}
\put(1,6){\circle*{.3}}
\put(1,0){\line(0,1)2}
\put(1,4){\line(0,1)2}
\put(1.4,-.15){$c_1=0$}
\put(1.4,.75){$\vdots$}
\put(1.4,1.85){$c_n$}
\put(.75,2.8){$P$}
\put(1.4,3.85){$c_{n+1}$}
\put(1.4,4.75){$\vdots$}
\put(1.4,5.85){$c_{2n}=1$}
\put(.2,-.95){{\rm Fig.\ 6}}
\end{picture}
\end{center}
\vspace*{4mm}
Let $P$ be a set and $\mathbf P_1=(P\times\{1\},\leq)$ and $\mathbf P_2=(P\times\{2\},\leq)$ posets. We call $\mathbf P_2$ the {\em dual} of $\mathbf P_1$ if for all $x,y\in P$ we have $(x,2)\leq(y,2)$ if and only if $(y,1)\leq(x,1)$. Similarly as before, a poset together with its dual can be extended to a residuated poset by means of a finite chain with at least four elements.

\begin{theorem}\label{th3}
Let $(P\times\{1\},\leq)$ be a poset, $(P\times\{2\},\leq)$ denote its dual, $n$ an integer $>1$, $k$ a non-negative integer, $0=c_1,\ldots,c_{2n+k}=1\notin P\times\{1,2\}$ and $Q:=(P\times\{1,2\})\cup\{c_1,\ldots,c_{2n+k}\}$. Extend $\leq$ to $Q$ by
\[
c_1<\cdots<c_n<(x,1)<c_{n+1}<\cdots<c_{n+k}<(y,2)<c_{n+k+1}<\cdots<c_{2n+k}
\]
for $x,y\in P$ and put $c_i':=c_{2n+k+1-i}$ for $i=1,\ldots,2n+k$ and $(x,i)':=(x,3-i)$ for $x\in P$ and $i\in\{1,2\}$. Define binary operations $\odot$ and $\rightarrow$ on $Q$ as follows:
\begin{align*}
& 0\odot x=x\odot0:=0, 1\odot x=x\odot1:=x, \\
& 0\rightarrow x=x\rightarrow1:=1, x\rightarrow0:=x', 1\rightarrow x:=x
\end{align*}
for $x\in Q$,
\begin{align*}
(x,i)\odot(y,j) & :=\left\{
\begin{array}{ll}
0   & \text{if }i=j=1\text{ or }((i,j)=(1,2)\text{ and }x\leq y)\text{ or} \\
    & (i,j)=(2,1)\text{ and }y\leq x, \\
c_2 & \text{otherwise} 
\end{array}
\right. \\
(x,i)\rightarrow(y,j) & :=\left\{
\begin{array}{ll}
1          & \text{if }(i=j=1\text{ and }x\leq y)\text{ or }(i,j)=(1,2)\text{ or} \\
           & (i=j=2\text{ and }y\leq x), \\
c_{2n+k-1} & \text{otherwise}
\end{array}
\right.
\end{align*}
for $x,y\in P$ and $i,j\in\{1,2\}$ and
\[
c_i\odot c_j:=\left\{
\begin{array}{ll}
0   & \text{if }i+j\leq2n+k+1, \\
c_2 & \text{otherwise} 
\end{array}
\right.
\quad\quad\quad c_i\rightarrow c_j:=\left\{
\begin{array}{ll}
1          & \text{if }i\leq j, \\
c_{2n+k-1} & \text{otherwise}
\end{array}
\right.
\]
\begin{align*}
& c_i\odot(x,1)=(x,1)\odot c_i:=c_i\odot c_{n+1}, c_i\rightarrow(x,1):=c_i\rightarrow c_n, (x,1)\rightarrow c_i:=c_{n+1}\rightarrow c_i, \\
& c_i\odot(x,2)=(x,2)\odot c_i:=c_i\odot c_{n+k+1}, c_i\rightarrow(x,2):=c_i\rightarrow c_{n+k}, \\
& (x,2)\rightarrow c_i:=c_{n+k+1}\rightarrow c_i
\end{align*}
for $i,j\in\{2,\ldots,2n+k-1\}$ and $x\in P$. Then $(Q,\leq,\odot,\rightarrow,1)$ is a residuated poset with an antitone involution $'$ satisfying $x'=x\rightarrow0$ for all $x\in Q$.
\end{theorem}

\begin{proof}
We apply Theorem~\ref{th1} to the poset $(P\cup\{c_3,\ldots,c_{2n+k-2}\},\leq,{}')$ with the antitone involution $'$. The element $c_3$ of Theorem~\ref{th1} corresponds to the element $c_{2n+k-1}$ of Theorem~\ref{th3}. According to Theorem~\ref{th1} we have for all $x\in Q$
\begin{align*}
& 0\odot x=x\odot0=0, 1\odot x=x\odot1=x, \\
& 0\rightarrow x=x\rightarrow1=1, x\rightarrow0=x', 1\rightarrow x=x.
\end{align*}
Moreover, we have for all $x,y\in P$ and $i,j\in\{1,2\}$
\begin{align*}
(x,i)\odot(y,j) & =\left\{
\begin{array}{ll}
0   & \text{if }(x,i)\leq(y,3-j), \\
c_2 & \text{otherwise} 
\end{array}
\right. \\
(x,i)\rightarrow(y,j) & =\left\{
\begin{array}{ll}
1          & \text{if }(x,i)\leq(y,j), \\
c_{2n+k-1} & \text{otherwise}
\end{array}
\right.
\end{align*}
Further, we have for all $i,j\in\{2,\ldots,2n+k-1\}$
\[
c_i\odot c_j=\left\{
\begin{array}{ll}
0   & \text{if }i\leq2n+k+1-j, \\
c_2 & \text{otherwise} 
\end{array}
\right.
\quad\quad\quad c_i\rightarrow c_j=\left\{
\begin{array}{ll}
1          & \text{if }i\leq j, \\
c_{2n+k-1} & \text{otherwise}
\end{array}
\right.
\]
Finally, we have for all $i\in\{2,\ldots,2n+k-1\}$ and $j\in\{1,2\}$
\[
c_i\odot(x,j)=(x,j)\odot c_i=\left\{
\begin{array}{ll}
0   & \text{if }c_i\leq(x,3-j), \\
c_2 & \text{otherwise} 
\end{array}
\right.
\]
\[
c_i\rightarrow(x,j)=\left\{
\begin{array}{ll}
1          & \text{if }c_i\leq(x,j), \\
c_{2n+k-1} & \text{otherwise}
\end{array}
\right.
\quad\quad\quad(x,j)\rightarrow c_i=\left\{
\begin{array}{ll}
1          & \text{if }(x,j)\leq c_i, \\
c_{2n+k-1} & \text{otherwise}
\end{array}
\right.
\]
\end{proof}

The poset $(Q,\leq)$ of Theorem~\ref{th2} is visualized in Figure~7:
\begin{center}
\setlength{\unitlength}{7mm}
\begin{picture}(2,10)
\put(1,0){\circle*{.3}}
\put(1,2){\circle*{.3}}
\put(1,4){\circle*{.3}}
\put(1,6){\circle*{.3}}
\put(1,8){\circle*{.3}}
\put(1,10){\circle*{.3}}
\put(1,0){\line(0,1)2}
\put(1,4){\line(0,1)2}
\put(1,8){\line(0,1)2}
\put(1.4,.75){$\vdots$}
\put(1.4,-.15){$c_1=0$}
\put(1.4,1.85){$c_n$}
\put(.75,2.8){$P$}
\put(1.4,3.85){$c_{n+1}$}
\put(1.4,4.75){$\vdots$}
\put(1.4,5.85){$c_{n+k}$}
\put(.75,6.8){$P^d$}
\put(1.4,7.85){$c_{n+k+1}$}
\put(1.4,8.75){$\vdots$}
\put(1.4,9.85){$c_{2n+k}=1$}
\put(.2,-.95){{\rm Fig.\ 7}}
\end{picture}
\end{center}
\vspace*{4mm}
It is obvious that the constructions described in Theorems~\ref{th2} and \ref{th3} can be generalized to finitely many posets and finite chains in between them. For instance, if we have posets $P_1,P_2,P_3$, their duals $P_1^d,P_2^d,P_3^d$ and a poset $P$ with an antitone involution then we can construct, similarly as before, a poset with an antitone involution of the form
\[
\cdots<x_1<\cdots<x_2<\cdots<x_3<\cdots<x<\cdots<y_3<\cdots<y_2<\cdots<y_1<\cdots
\]
for all $x_1\in P_1,x_2\in P_2,x_3\in P_3,x\in P,y_3\in P_3^d,y_2\in P_2^d,y_1\in P_1^d$ where the dots (from left to right) indicate finite chains with $r,s,t,u,u,t,s,r$ elements, respectively, and then we can apply Theorem~\ref{th1} to this poset with antitone involution.

The following result is well known, see e.g.\ \cite B or \cite{WD}.

\begin{lemma}\label{lem2}
Let $(B,\vee,\wedge,{}',p,q)$ be a Boolean algebra and put
\begin{align*}
      x\odot y & :=x\wedge y, \\
x\rightarrow y & :=x'\vee y
\end{align*}
for every $x,y\in B$. Then $(B,\leq,\odot,\rightarrow,q)$ is a residuated poset.
\end{lemma}

In the following we extend a Boolean algebra $(B,\vee,\wedge,{}',p,q)$ to a residuated lattice $(Q,\vee,\wedge,\odot,\rightarrow,1)$
with an antitone involution $x'=x\rightarrow0$ by means of a (finite) chain such that the operations $\odot$ and $\rightarrow$ coincide on $B$ with those mentioned in Lemma~\ref{lem2}.

\begin{theorem}\label{th5}
Let $(B,\vee,\wedge,{}',p,q)$ be a Boolean algebra, $n$ a positive integer, $0=c_1,\ldots$ $\ldots,c_{2n}=1\notin B$ and $Q:=B\cup\{c_1,\ldots,c_{2n}\}$. Extend $\leq$ and $'$ from $B$ to $Q$ by
\[
c_1<\cdots<c_n<x<c_{n+1}<\cdots<c_{2n}
\]
for $x\in B$ and put $c_i':=c_{2n+1-i}$ for $i=1,\ldots,2n$. Define binary operations $\odot$ and $\rightarrow$ on $Q$ as follows:
\[
x\odot y:=\left\{
\begin{array}{ll}
0         & \text{if }x\leq y', \\
x\wedge y & \text{otherwise} 
\end{array}
\right.
\quad\quad\quad x\rightarrow y:=\left\{
\begin{array}{ll}
1        & \text{if }x\leq y, \\
x'\vee y & \text{otherwise}
\end{array}
\right.
\]
for $x,y\in Q$. Then $(Q,\leq,\odot,\rightarrow,1)$ is a residuated lattice with an antitone involution $x'=x\rightarrow0$ where for all $x,y\in B$ we have
\begin{align*}
      x\odot y & =x\wedge y, \\
x\rightarrow y & =x'\vee y.
\end{align*}
\end{theorem}

\begin{proof}
Let $a,b,c\in Q$, $d,e\in B$ and $i,j,k\in\{1,\ldots,2n\}$. Within this proof
\begin{align*}
& a<B\text{ means }a<x\text{ for all }x\in B, \\
& a>B\text{ means }a>x\text{ for all }x\in B, \\
& i\vee j:=\max(i,j), \\
& i\wedge j:=\min(i,j).
\end{align*}
Obviously, $(Q,\leq,{}',0,1)$ is a bounded lattice with an antitone involution. We have
\begin{align*}
      a\odot b & =b\odot a\text{ since }a\leq b'\text{ is equivalent to }b\leq a', \\
      0\odot a & =0, \\
      1\odot a & =\left\{
\begin{array}{ll}
0=a         & \text{if }a=0, \\
1\wedge a=a & \text{otherwise}
\end{array}
\right. \\
0\rightarrow a & =a\rightarrow1=1, \\
 a\rightarrow0 & =\left\{
\begin{array}{ll}
1=0'=a'    & \text{if }a=0, \\
a'\vee0=a' & \text{otherwise}
\end{array}
\right. \\
      d\odot e & =\left\{
\begin{array}{ll}
c_1=d\wedge e & \text{if }d\leq e', \\
d\wedge e     & \text{otherwise}
\end{array}
\right. \\
d\rightarrow e & =\left\{
\begin{array}{ll}
1=d'\vee e & \text{if }d\leq e, \\
d'\vee e   & \text{otherwise}.
\end{array}
\right.
\end{align*}
This shows
\begin{align*}
& 0\odot x=x\odot0=0, 1\odot x=x\odot1=x, \\
& 0\rightarrow x=x\rightarrow1=1, x\rightarrow0=x', 1\rightarrow x=x
\end{align*}
($x\in Q$) and
\begin{align*}
      x\odot y & =x\wedge y, \\
x\rightarrow y & =x'\vee y
\end{align*}
for all $x,y\in B$. \\
If $a\leq b'$ and $b\leq c'$ then $(a\odot b)\odot c=0\odot c=0=a\odot0=a\odot(b\odot c)$. \\
If $a\leq b'$ and $b\not\leq c'$ then $(a\odot b)\odot c=0\odot c=0=a\odot(b\wedge c)=a\odot(b\odot c)$. \\
If $a\not\leq b'$ and $b\leq c'$ then $(a\odot b)\odot c=(a\wedge b)\odot c=0=a\odot0=a\odot(b\odot c)$. \\
Now consider the case $a\not\leq b'$ and $b\not\leq c'$. Then
\begin{align*}
(a\odot b)\odot c & =(a\wedge b)\odot c, \\
 a\odot(b\odot c) & =a\odot(b\wedge c).
\end{align*}
If $a,b,c\in B$ then $(a\wedge b)\odot c=(a\wedge b)\wedge c=a\wedge(b\wedge c)=a\odot(b\wedge c)$. \\
If $a,b\in B$ and $c\notin B$ then $c>B$ and hence $(a\wedge b)\odot c=(a\wedge b)\wedge c=a\wedge b=a\odot b=a\odot(b\wedge c)$. \\
If $a\in B$, $b\notin B$ and $c\in B$ then $b>B$ and hence $(a\wedge b)\odot c=a\odot c=a\odot(b\wedge c)$. \\
If $a\in B$ and $b,c\notin B$ then $b>B$ and ($a\leq c'$ if and only if $a\leq b'\vee c'$). Hence $(a\wedge b)\odot c=a\odot c=0=a\odot(b\wedge c)$ if $a\leq c'$ and $(a\wedge b)\odot c=a\odot c=a\wedge c=(a\wedge b)\wedge c=a\wedge(b\wedge c)=a\odot(b\wedge c)$ otherwise. \\
If $a\notin B$ and $b,c\in B$ then $a>B$ and hence $(a\wedge b)\odot c=b\odot c=b\wedge c=(a\wedge b)\wedge c=a\wedge(b\wedge c)=a\odot(b\wedge c)$. \\
If $a\notin B$, $b\in B$ and $c\notin B$ then $a,c>B$ and hence $(a\wedge b)\odot c=b\odot c=b\wedge c=b=a\wedge b=a\odot b=a\odot(b\wedge c)$. \\
If $a,b\notin B$ and $c\in B$ then $b>B$ and ($a\leq c'$ if and only if $a\wedge b\leq c'$). Hence $(a\wedge b)\odot c=0=a\odot c=a\odot(b\wedge c)$ if $a\leq c'$ and $(a\wedge b)\odot c=(a\wedge b)\wedge c=a\wedge(b\wedge c)=a\wedge c=a\odot c=a\odot(b\wedge c)$ otherwise. \\
Finally, assume $a,b,c\notin B$. Without loss of generality, $(a,b,c)=(c_i,c_j,c_k)$. Now we have
\[
c_i\odot c_j=\left\{
\begin{array}{ll}
0             & \text{if }i+j\leq 2n+1, \\
c_{i\wedge j} & \text{otherwise}
\end{array}
\right.
\]
If $i+j,j+k\leq2n+1$ then $(c_i\odot c_j)\odot c_k=0\odot c_k=0=c_i\odot0=c_i\odot(c_j\odot c_k)$. \\
If $i+j\leq2n+1<j+k$ then $i+(j\wedge k)=(i+j)\wedge(i+k)\leq i+j\leq2n+1$ and hence $(c_i\odot c_j)\odot c_k=0\odot c_k=0=c_i\odot c_{j\wedge k}=c_i\odot(c_j\odot c_k)$. \\
If $j+k\leq2n+1<i+j$ then $(i\wedge j)+k=(i+k)\wedge(j+k)\leq j+k\leq2n+1$ and hence $(c_i\odot c_j)\odot c_k=c_{i\wedge j}\odot c_k=0=c_i\odot0=c_i\odot(c_j\odot c_k)$. \\
If $i+j,j+k>2n+1$ then
\begin{align*}
(i\wedge j)+k=(i+k)\wedge(j+k)\leq2n+1 & \text{ if and only if }i+k\leq2n+1, \\
i+(j\wedge k)=(i+j)\wedge(i+k)\leq2n+1 & \text{ if and only if }i+k\leq2n+1
\end{align*}
and hence $(c_i\odot c_j)\odot c_k=c_{i\wedge j}\odot c_k=c_i\odot c_{j\wedge k}=c_i\odot(c_j\odot c_k)$. \\
This shows that $\odot$ is associative. Since $a\leq b'$ is equivalent to $b\leq a'$, $\odot$ is commutative. Therefore, $(Q,\odot,1)$ is a commutative monoid. \\
If $a\leq b'$ and $b\leq c$ then $a\odot b=0\leq c$ and $a\leq1=b\rightarrow c$. \\
If $a\leq b'$ and $b\not\leq c$ then $a\odot b=0\leq c$, and $a\leq b\rightarrow c$ since $a\leq b'\vee c$. \\
If $a\not\leq b'$ and $b\leq c$ then $a\odot b\leq c$ since $a\wedge b\leq c$, and $a\leq1=b\rightarrow c$. \\
Now assume $a\not\leq b'$ and $b\not\leq c$. \\
If $a,b,c\in B$ then $a\wedge b\leq c$ if and only if $a\leq b'\vee c$ because of Lemma~\ref{lem2}. \\
If $a,b\in B$ and $c\notin B$ then $c<B$ and hence $a\wedge b\not\leq c$ and $a\not\leq b'=b'\vee c$. \\
If $b\notin B$ and $c\in B$ then $b>B$ and hence $a\wedge b\leq c$ and $a\leq b'\vee c$ are both equivalent to $a\leq c$. \\
If $a\in B$ and $b,c\notin B$ then $b>B$ and hence $a\wedge b\leq c$ and $a\leq b'\vee c$ are both equivalent to $a\leq c$. \\
If $a\notin B$ and $b,c\in B$ then $a>B$ and hence $a\wedge b=b\not\leq c$ and $a\not\leq b'\vee c$. \\
If $a\notin B$, $b\in B$ and $c\notin B$ then $c<B<a$ and hence $a\wedge b=b\not\leq c$ and $a\not\leq b'=b'\vee c$. \\
Finally, assume $a,b,c\notin B$. Without loss of generality, $(a,b,c)=(c_i,c_j,c_k)$. Now we have
\[
c_i\rightarrow c_j=\left\{
\begin{array}{ll}
1                  & \text{if }i\leq j, \\
c_{(2n+1-i)\vee j} & \text{otherwise}
\end{array}
\right.
\]
If $i+j\leq2n+1$ and $j\leq k$ then $c_i\odot c_j=0\leq c_k$ and $c_i\leq1=c_j\rightarrow c_k$. \\
If $i+j\leq2n+1$ and $j>k$ then $c_i\odot c_j=0\leq c_k$ and $c_i\leq c_{(2n+1-j)\vee k}$ since $i\leq2n+1-j$. \\
If $i+j>2n+1$ and $j\leq k$ then $c_i\odot c_j=c_{i\wedge j}\leq c_k$ and $c_i\leq c_{2n}=c_j\rightarrow c_k$. \\
Finally, assume $i+j>2n+1$ and $j>k$. Then the following are equivalent:
\begin{align*}
 c_i\odot c_j & \leq c_k, \\
c_{i\wedge j} & \leq c_k, \\
    i\wedge j & \leq k, \\
            i & \leq k.
\end{align*}
Moreover, the following are equivalent:
\begin{align*}
c_i & \leq c_j\rightarrow c_k, \\
c_i & \leq c_{(2n+1-j)\vee k}, \\
  i & \leq(2n+1-j)\vee k, \\
i+j & \leq(2n+1)\vee(k+j), \\
i+j & \leq k+j, \\
  i & \leq k.
\end{align*}
This proves adjointness.
\end{proof}

The poset $(Q,\leq)$ of Theorem~\ref{th2} is visualized in Figure~8:
\begin{center}
\setlength{\unitlength}{7mm}
\begin{picture}(2,8)
\put(1,0){\circle*{.3}}
\put(1,2){\circle*{.3}}
\put(1,3){\circle*{.3}}
\put(1,5){\circle*{.3}}
\put(1,6){\circle*{.3}}
\put(1,8){\circle*{.3}}
\put(1,4){\circle{2}}
\put(1,0){\line(0,1)3}
\put(1,5){\line(0,1)3}
\put(1.4,-.15){$c_1=0$}
\put(1.4,.75){$\vdots$}
\put(1.4,1.85){$c_n$}
\put(.75,3.8){$B$}
\put(1.4,5.85){$c_{n+1}$}
\put(1.4,6.75){$\vdots$}
\put(1.4,7.85){$c_{2n}=1$}
\put(.2,-.95){{\rm Fig.\ 8}}
\end{picture}
\end{center}
\vspace*{4mm}

\begin{example}
We consider the special case where $(B,\vee,\wedge,{}',p,q)$ is the eight-element Boolean algebra and $n=2$. The poset $(Q,\leq)$ is visualized in Figure~9:
\begin{center}
\setlength{\unitlength}{7mm}
\begin{picture}(6,14)
\put(3,0){\circle*{.3}}
\put(3,2){\circle*{.3}}
\put(3,4){\circle*{.3}}
\put(1,6){\circle*{.3}}
\put(3,6){\circle*{.3}}
\put(5,6){\circle*{.3}}
\put(1,8){\circle*{.3}}
\put(3,8){\circle*{.3}}
\put(5,8){\circle*{.3}}
\put(3,10){\circle*{.3}}
\put(3,12){\circle*{.3}}
\put(3,14){\circle*{.3}}
\put(3,4){\line(-1,1)2}
\put(3,0){\line(0,1)6}
\put(3,4){\line(1,1)2}
\put(3,6){\line(-1,1)2}
\put(3,6){\line(1,1)2}
\put(3,8){\line(-1,-1)2}
\put(3,8){\line(1,-1)2}
\put(3,10){\line(-1,-1)2}
\put(3,8){\line(0,1)6}
\put(3,10){\line(1,-1)2}
\put(1,6){\line(0,1)2}
\put(5,6){\line(0,1)2}
\put(3.4,-.15){$0$}
\put(3.4,1.85){$c_2$}
\put(3.4,3.85){$p$}
\put(.3,5.85){$a$}
\put(3.4,5.85){$b$}
\put(5.4,5.85){$c$}
\put(.3,7.85){$c'$}
\put(3.4,7.85){$b'$}
\put(5.4,7.85){$a'$}
\put(3.4,9.85){$q$}
\put(3.4,11.85){$c_3$}
\put(3.4,13.85){$1$}
\put(2.2,-1){{\rm Fig.\ 9}}
\end{picture}
\end{center}
\vspace*{4mm}
The corresponding residuated lattice has the operations
\[
\begin{array}{l|llllllllllll}
\odot & 0 & c_2 & p & a & b & c & a' & b' & c' & q  & c_3 & 1 \\
\hline
0     & 0 & 0   & 0 & 0 & 0 & 0 & 0  & 0  & 0  & 0  & 0   & 0 \\
c_2   & 0 & 0   & 0 & 0 & 0 & 0 & 0  & 0  & 0  & 0  & 0   & c_2 \\
p     & 0 & 0   & 0 & 0 & 0 & 0 & 0  & 0  & 0  & 0  & p   & p \\
a     & 0 & 0   & 0 & a & 0 & 0 & 0  & a  & a  & a  & a   & a \\
b     & 0 & 0   & 0 & 0 & b & 0 & b  & 0  & b  & b  & b   & b \\
c     & 0 & 0   & 0 & 0 & 0 & c & c  & c  & 0  & c  & c   & c \\
a'    & 0 & 0   & 0 & 0 & b & c & a' & c  & b  & a' & a'  & a' \\
b'    & 0 & 0   & 0 & a & 0 & c & c  & b' & a  & b' & b'  & b' \\
c'    & 0 & 0   & 0 & a & b & 0 & b  & a  & c' & c' & c'  & c' \\
q     & 0 & 0   & 0 & a & b & c & a' & b' & c' & q  & q   & q \\
c_3   & 0 & 0   & p & a & b & c & a' & b' & c' & q  & c_3 & c_3 \\
1     & 0 & c_2 & p & a & b & c & a' & b' & c' & q  & c_3 & 1
\end{array}
\]
\[
\begin{array}{l|llllllllllll}
\rightarrow & 0   & c_2 & p  & a  & b  & c  & a' & b' & c' & q & c_3 & 1 \\
\hline
0           & 1   & 1   & 1  & 1  & 1  & 1  & 1  & 1  & 1  & 1 & 1   & 1 \\
c_2         & c_3 & 1   & 1  & 1  & 1  & 1  & 1  & 1  & 1  & 1 & 1   & 1 \\
p           & q   & q   & 1  & 1  & 1  & 1  & 1  & 1  & 1  & 1 & 1   & 1 \\
a           & a'  & a'  & a' & 1  & a' & a' & a' & 1  & 1  & 1 & 1   & 1 \\
b           & b'  & b'  & b' & b' & 1  & b' & 1  & b' & 1  & 1 & 1   & 1 \\
c           & c'  & c'  & c' & c' & c' & 1  & 1  & 1  & c' & 1 & 1   & 1 \\
a'          & a   & a   & a  & a  & c' & b' & 1  & b' & c' & 1 & 1   & 1 \\
b'          & b   & b   & b  & c' & b  & a' & a' & 1  & c' & 1 & 1   & 1 \\
c'          & c   & c   & c  & b' & a' & c  & a' & b' & 1  & 1 & 1   & 1 \\
q           & p   & p   & p  & a  & b  & c  & a' & b' & c' & 1 & 1   & 1 \\
c_3         & c_2 & c_2 & p  & a  & b  & c  & a' & b' & c' & q & 1   & 1 \\
1           & 0   & c_2 & p  & a  & b  & c  & a' & b' & c' & q & c_3 & 1
\end{array}
\]
\end{example}

Authors' addresses:

Ivan Chajda \\
Palack\'y University Olomouc \\
Faculty of Science \\
Department of Algebra and Geometry \\
17.\ listopadu 12 \\
771 46 Olomouc \\
Czech Republic \\
ivan.chajda@upol.cz

Miroslav Kola\v r\'ik \\
Palack\'y University Olomouc \\
Faculty of Science \\
Department of Computer Science \\
17.\ listopadu 12 \\
771 46 Olomouc \\
Czech Republic \\
miroslav.kolarik@upol.cz

Helmut L\"anger \\
TU Wien \\
Faculty of Mathematics and Geoinformation \\
Institute of Discrete Mathematics and Geometry \\
Wiedner Hauptstra\ss e 8-10 \\
1040 Vienna \\
Austria, and \\
Palack\'y University Olomouc \\
Faculty of Science \\
Department of Algebra and Geometry \\
17.\ listopadu 12 \\
771 46 Olomouc \\
Czech Republic \\
helmut.laenger@tuwien.ac.at
\end{document}